\documentclass[11pt]{amsart}
\usepackage{amssymb,amsmath,amsthm,newlfont,enumerate, graphicx}

\theoremstyle{plain}
\newtheorem{theorem}{Theorem}
\newtheorem{lemma}{Lemma}
\theoremstyle{remark}
\newtheorem{remark}{Remark}
\newtheorem{example}{Example}

\newcommand{\CC}{{\mathbb C}}
\newcommand{\RR}{{\mathbb R}}

\newcommand{\NN}{{\mathbb N}}
\newcommand{\DD}{{\mathbb D}}
\newcommand{\EE}{{\mathbb E}}
\newcommand{\TT}{{\mathbb T}}

\newcommand{\Dveca}{{\mathfrak D}_{\vec{\alpha}}}
\newcommand{\Dvecaa}{{\mathfrak D}_{(\alpha_1,\alpha_2)}}
\newcommand{\veca}{{\vec{\alpha}}}
\newcommand{\la}{\lambda}
\newcommand{\al}{\alpha}

\renewcommand{\Re}{\operatorname{Re}}

\newcommand{\McC}{\raise.5ex\hbox{c}}

\begin{document}

\date{\today}

\title[Cyclic polynomials in Dirichlet spaces]{Cyclic polynomials in anisotropic Dirichlet~spaces}
\author[Knese]{Greg Knese}
\address{Department of Mathematics, Washington University in St.\ Louis, 
One Brookings Drive, Campus Box 1146, St. Louis, MO 63130, U.S.A.}
\email{geknese@math.wustl.edu}

\author[Kosi\'nski]{\L ukasz Kosi\'nski}
\address{Institute of Mathematics, Jagiellonian University, \L ojasiewicza 6, 30-348 Krak\'ow, Poland.}
\email{lukasz.kosinski@im.uj.edu.pl}

\author[Ransford]{Thomas J. Ransford}
\address{D\'epartement de math\'ematiques et de statistique, Universit\'e Laval, 
Pavillon Alexandre-Vachon 1045, Avenue de la M\'edecine, Qu\'ebec City, Qu\'ebec, G1V 0A6, Canada.}
\email{thomas.ransford@mat.ulaval.ca}

\author[Sola]{Alan A. Sola}
\address{Department of Mathematics and Statistics, University of South Florida, 4202 E Fowler Avenue, CMC342, Tampa, FL 33620, U.S.A.}
\email{sola@usf.edu}

\thanks{GK supported by NSF grant DMS-1363239.}
\thanks{\L K supported by NCN grant UMO-2014/15/D/ST1/01972. }
\thanks{TR supported by grants from NSERC and the Canada research chairs program.}

\begin{abstract}
  Consider the   Dirichlet-type space on the bidisk consisting~of~holomorphic functions $f(z_1,z_2):=\sum_{k,l\ge0}a_{kl}z_1^kz_2^l$  \phantom{a}such that
   $\sum_{k,l\ge0}(k+1)^{\alpha_1}(l+1)^{\alpha_2}|a_{kl}|^2<\infty.$ Here the parameters $\alpha_1,\alpha_2$ are arbitrary real numbers. We characterize the polynomials that are cyclic for the shift operators on this space. More precisely, we show that, given an irreducible polynomial $p(z_1,z_2)$ depending on both $z_1$ and $z_2$ and having no zeros in the bidisk:
  \begin{itemize}
  \item if $\alpha_1+\alpha_2\le1$, then $p$ is cyclic;
  \item if $\alpha_1+\alpha_2>1$ and $\min\{\alpha_1,\alpha_2\}\le1$, then $p$ is cyclic if and only if it has finitely many zeros in the two-torus $\TT^2$;
  \item if $\min\{\alpha_1,\alpha_2\}>1$, then $p$ is cyclic if and only if it has no zeros in $\TT^2$.
  \end{itemize}  \end{abstract}

\maketitle

\section{Introduction}

Spaces of analytic functions on the unit bidisk
\[\DD^2=\{z=(z_1, z_2)\in \CC^2\colon |z_1|<1, |z_2|<1\}\] 
provide a compelling meeting place for the study of function theory
and operator theory.  Bounded analytic functions on the bidisk and
their relation to inequalities on pairs of commuting operators provide
a well-developed example of this (see the survey \cite{McCarthy}).
Our focus is on certain Hilbert spaces of analytic functions on
$\DD^2$, the anisotropic Dirichlet spaces, and a perennial topic in
operator theory: understanding the cyclic vectors for model operators,
in this case the coordinate shifts.  In full generality, this is
likely a difficult problem, but if we restrict ourselves to analyzing
cyclic \emph{polynomials} we can give a full characterization while
simultaneously learning much about the behavior of two variable stable
polynomials; i.e. those with no zeros on the bidisk.  A crucial role
is played by the size of a polynomial's zero set on the $2$-torus
\[\TT^2=\{\zeta=(\zeta_1, \zeta_2)\in \CC^2\colon |\zeta_1|=1,
|\zeta_2|=1\},\]
the distinguished boundary of the unit bidisk.

\subsection{Dirichlet spaces on the bidisk}

Let $\veca=(\alpha_1, \alpha_2)\in \RR^2$
be fixed. We say that a holomorphic function $f\colon \DD^2\to \mathbb{C}$ having power series expansion \[f(z_1,z_2)=\sum_{k\ge0}\sum_{l\geq 0}a_{k,l}z_1^kz_2^l\] belongs to the {\it anisotropic weighted Dirichlet space} $\Dveca$ if
\begin{equation}
\|f\|_{\veca}^2:=\sum_{k\ge 0}\sum_{l\ge 0}(k+1)^{\alpha_1}(l+1)^{\alpha_2}|a_{k,l}|^2<\infty.
\label{D: Dvecadef}
\end{equation}
These spaces have been considered by a number of mathematicians, see for instance \cite{Hed88, Kap94, JR06, BKKLSS14}. We refer the reader to these papers, and the references therein, for further 
background material, and only give a brief summary of some facts we shall need later on. As is pointed out in these references, for $\veca\in \RR^2$ with $\alpha_1<2$ and $\alpha_2<2$, the spaces can be furnished with the equivalent norm
\begin{equation}
\|f\|^2_{\veca,\ast}=|f(0,0)|^2+\mathfrak{D}_{\veca}(f),
\label{integralnorm}
\end{equation}
where
\begin{align*}
\mathfrak{D}_{\veca}(f)=&\int_{\mathbb{D}}|\partial_{z_1}[f(z_1,0)]|^2dA_{\alpha_1}(z_1)+
\int_{\mathbb{D}}|\partial_{z_2}[f(0,z_2)]|^2dA_{\alpha_2}(z_2)\\&+
\int_{\mathbb{D}^2}|\partial_{z_2}\partial_{z_1}f(z_1,z_2)|^2dA_{\alpha_1}(z_1)dA_{\alpha_2}(z_2).
\label{normintegrals}
\end{align*}
Here, for $k=1,2$, we have set $dA_{\alpha_k}(z_k)=(1-|z_k|^2)^{1-\alpha_k}dA(z_k)$, where $dA(z)=\pi^{-1}dxdy$ denotes normalized area measure. More compactly, we have
\[\|f\|^2_{\veca,\ast}=\int_{\DD^2}|\partial_{z_2}\partial_{z_1}(z_1z_2f(z_1,z_2))|^2\,dA_{\alpha_1}(z_1)\,dA_{\alpha_2}(z_2).\]

For all choices of $\veca\in \RR^2$, polynomials in two complex variables form a dense subset of $\Dveca$. Moreover, the classical one-variable weighted Dirichlet spaces $D_{\alpha}$, consisting of analytic functions $f=\sum_{k\geq 0}a_kz^k$ on the unit disk $\DD$ having 
\[\|f\|_{\alpha}=\sum_{k\ge0}(k+1)^{\alpha}|a_k|^2<\infty,\]
embed in $\Dveca$ in a natural way. These one-variable spaces are discussed 
in the textbook \cite{EKMRBook}, and again admit an equivalent integral norm,
\[\|f\|_{\alpha, \ast}^2=|f(0,0)|^2+\int_{\DD}|f'(z)|^2(1-|z|^2)^{1-\alpha}dA(z).\] 

If $\min\{\alpha_1,\alpha_2\}>1$, then $\Dveca$ is a Banach algebra of continuous functions on the closed unit bidisk; this can be seen from the Cauchy-Schwarz inequality and the convergence of the series
$\sum_{k,l\geq 0}(k+1)^{-\alpha_1}(l+1)^{-\alpha_2}$. Setting $\alpha_1=\alpha_2=0$, we are led to the Hardy space $H^2(\DD^2)$ of the bidisk, which was studied by Rudin in the 60s, see 
\cite{RudBook}. 
The parameter choice $\alpha_1=\alpha_2=-1$ yields the Bergman space
of the bidisk. The choice $\alpha_1=\alpha_2=1$ corresponds to the Dirichlet space of
the bidisk, which can be characterized by the fact that
pre-composition with automorphisms of the bidisk form a set of unitary
operators.  This space was considered by Kaptano\u{g}lu \cite{Kap94},
among others.

Isotropic weighted Dirichlet spaces, the cases with 
$\alpha_1=\alpha_2$, were recently studied in depth in \cite{BCLSS13II, BKKLSS14}; we shall use $\mathfrak{D}_{\alpha}$ to denote these isotropic spaces. The anisotropic spaces $\Dveca$ were 
studied by Jupiter and Redett \cite{JR06}, who consider order relations and identify multipliers between different $\Dveca$. For instance, they observe that $\Dveca \subset \mathfrak{D}_{\vec{\beta}}$ if 
$\alpha_1\geq\beta_1$ and $\alpha_2\geq \beta_2$. A related fact that
we shall use frequently is that 
\begin{equation}
f\in \Dveca \quad \textrm{if and only if} \quad \partial_{z_1}f\in
\mathfrak{D}_{\alpha_1-2, \alpha_2} \text{ and } f(0,\cdot) \in D_{\al_2}
\label{derivativerelation}
\end{equation}
and similarly for $\partial_{z_2}f$. A {\it multiplier} of $\Dveca$ is a function $\phi \colon \DD^2\to \CC$ that is holomorphic and satisfies $\phi f\in \Dveca$ for every $f\in \Dveca$.
In the case of the Hardy and Bergman spaces, the multipliers are
precisely the bounded analytic functions, but for general $\Dveca$ it
is not as easy to describe the multiplier space $M(\Dveca)$.  For our
purposes it will suffice to note that any function that is analytic on
a neighborhood of the closed bidisk is a multiplier on each $\Dveca$.
\subsection{Shift operators and cyclic vectors}
Consider the two linear operators $S_1,S_2\colon \Dveca \to \Dveca$ defined via 
\begin{equation}
S_1\colon f \mapsto z_1\cdot f \quad \textrm{and}
\quad S_2\colon f\mapsto z_2\cdot f.
\end{equation}
When viewed as acting on the coefficient matrix of a function $f$, the operators $S_1$ and $S_2$ become right and upwards translations, justifying the designation {\it shift operators}. It is clear that the coordinate shifts commute, and, in view of \eqref{D: Dvecadef}, they are bounded on $\Dveca$. 

Coordinate shifts $\{S_1, S_2\}$ acting on $\Dveca$ furnish a natural model of multivariate operator theory. The structure of invariant subspaces of these operators is still rather poorly understood, even in the case of the unweighted Hardy space $H^2=\mathfrak{D}_{\vec{0}}$, where the shifts are commuting isometries. See Rudin's book \cite{RudBook} for some basic results and pathologies such as the existence of an invariant subspace containing no 
bounded elements, as well as \cite{Man88, Chatetal14} and the references therein for some positive results, such as a conditional version of Beurling's theorem on invariant subspaces.

It is easy to exhibit invariant subspaces of $\Dveca$. For instance, fixing a function $f\in \Dveca$, we can form the {\it cyclic subspace}
\begin{equation}
[f]_{\veca}=\mathrm{clos}\,\,\mathrm{span}\{z_1^kz_2^lf\colon k,l\geq 0\}
\label{E: cyclicsubspace}
\end{equation}
which is invariant under $\{S_1,S_2\}$ by definition; the closure 
is taken with respect to the $\Dveca$ norm. Another class of 
invariant subspaces is given by {\it zero-based subspaces}: a simple example 
is the subspace of functions divisible by the polynomial $f=z_2$.

It is a much more difficult task to obtain a concrete description of general invariant subspaces of $\Dveca$ and their elements, even in the simplest case of cyclic subspaces. In this paper, we are primarily interested in identifying {\it cyclic vectors} for the coordinate shifts: functions $f\in \Dveca$ such that $[f]_{\vec{\alpha}}=\Dveca$. This seems like a hard problem for general functions, and we restrict our attention to the case where $f$ itself is a polynomial in two variables. In what follows, we shall use the letter $p$ to indicate that we are dealing with a fixed polynomial. In that setting, we are able to give a complete characterization, extending the corresponding result in \cite{BKKLSS14} to the anisotropic setting.

The cyclicity of a function $f\in \Dveca$ is intimately connected with its vanishing properties. The constant function $p(z_1,z_2)=1$ is cyclic in all $\Dveca$; this is just a reformulation of the fact that polynomials are dense. Similarly, functions that are holomorphic on a neighborhood of the bidisk, and are non-vanishing on $\overline{\DD^2}$ are cyclic 
for all $\Dveca$. At the other extreme, the polynomial $p(z_1,z_2)=z_2$ is clearly not cyclic, as elements of $[z_2]_{\veca}$ have to vanish on the set $\{z_2=0\}\cap \DD^2$. More generally, as a 
consequence of boundedness of point evaluation functionals on $\Dveca$ (see \cite{JR06}), no function that vanishes on the interior of the bidisk can be cyclic. The case of zeros on the boundary is subtler, and polynomials that vanish on the boundary of the bidisk remain cyclic provided their zero sets in the torus are not too large, relative to the parameter $\veca\in \RR^2$.

\subsection{Statement of results}
In \cite{BKKLSS14}, a complete classification of cyclic polynomials in
the isotropic spaces $\mathfrak{D}_{\alpha}$ was found in terms of
conditions on $\mathcal{Z}(p)\cap \TT^2$, where
$\mathcal{Z}(p)=\{z\in \CC^2\colon p(z)=0\}$ is the zero set of a
polynomial $p=p(z_1,z_2)$.  Earlier, Neuwirth, Ginsberg, and Newman
\cite{NGN70} had shown that all polynomials that do not vanish in
$\DD^n$ are cyclic in $H^2(\DD^n)$, and hence in all $\Dveca$ that 
contain $H^2(\mathbb{D}^2)$; see also Gelca's paper \cite{G95}.

The purpose of this paper is to extend the classification result of \cite{BKKLSS14} to the anisotropic setting. 
Part of the proof in that paper was based on $\alpha$-capacities and Cauchy integrals, and does not 
generalize in an obvious way to the general setting where one of the components in $\veca$ may be negative. At the same time, we show how the arguments in \cite{BKKLSS14} which relied on prior work of Knese and others (see \cite{K10a, K10b} and the references in those papers) on polynomials having 
determinantal representations can be replaced by integral estimates. While the former theory is elegant, our approach is more direct, and applies in the case of negative parameters as well; cf \cite[Theorem 3.1]{BKKLSS14}, where it is assumed that $\alpha>0$.

We now state the main result of this paper.
\begin{theorem}\label{T:maintheorem}
Let $p$ be an irreducible polynomial, depending on both $z_1$ and $z_2$, 
with no zeros in the bidisk.
\begin{enumerate}
\item If $\alpha_1+\alpha_2 \leq 1$ then $p$ is cyclic in $\mathfrak{D}_{\vec{\alpha}}$.
\item If $\alpha_1+\alpha_2>1$ and $\min\{\alpha_1, \alpha_2\}\leq 1$, then $p$ is cyclic in $\mathfrak{D}_{\vec{\alpha}}$ if and only if $\mathcal{Z}(p)\cap \mathbb{T}^2$ is empty or finite.
\item If $\min\{\alpha_1,\alpha_2\}>1,$ then $p$ is cyclic in $\mathfrak{D}_{\vec{\alpha}}$ if and only if 
$\mathcal{Z}(p)\cap \mathbb{T}^2$ is empty.
\end{enumerate}
\end{theorem}

The parameter regions in Theorem \ref{T:maintheorem} are illustrated in Figure \ref{alphaplanefig}. 
\begin{figure}
\includegraphics[width=0.4 \textwidth]{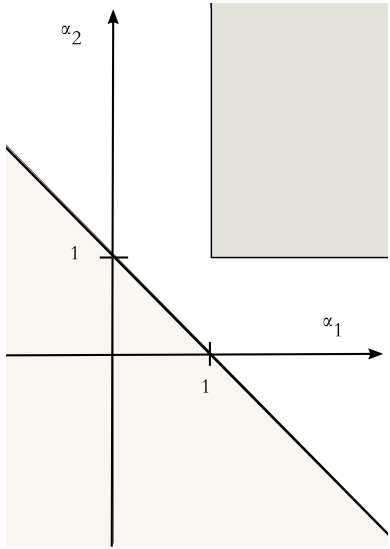}
\caption{Shaded region, lower left: parameter values $\veca$ for which {\it all} polynomials that do not vanish in $\DD^2$ are cyclic; Shaded region, top right: parameter values $\veca$ for which {\it no} polynomials that vanish in $\overline{\DD^2}$ are cyclic.}
\label{alphaplanefig}
\end{figure}
As is explained in \cite{BKKLSS14}, the set $\mathcal{Z}(p)\cap \TT^2$
associated with an irreducible polynomial $p$ is either a curve or a
finite number of points. The requirement that $f$ be irreducible is
not a serious restriction.  Indeed, since all polynomials are multipliers, a
product of polynomials is cyclic precisely when all its factors are.
In the formulation of the theorem the assumption that $f$ depends on
both variables is made to avoid complications created by one variable
polynomials.  For instance, the function $f(z) = 1-z_j$ is cyclic in $\Dvecaa$ precisely when
$\alpha_j \leq 1$.

\begin{remark}\label{R:tensorrem}
 A new feature that appears in the anisotropic 
framework is that the degrees of smoothness/roughness of the space in
the two coordinate directions interact to determine whether a
polynomial is cyclic or not.
For instance, according to Theorem \ref{T:maintheorem}, the two-variable polynomial $p=1-z_1z_2$ is cyclic in $\mathfrak{D}_{(-2,2)}$, as is the one-variable polynomial $p=1-z_1$. By contrast, the function $p=1-z_2$ is {\it not} cyclic in $\mathfrak{D}_{(-2,2)}$ because $1-z$ is not cyclic in $D_{2}$. 
\end{remark}

\section{Preliminaries}\label{S:examples}
In this section, we first record some facts concerning one-variable functions in $D_{\alpha}$. Both the results and the arguments used to establish them will be used in our subsequent two-variable proofs. We then discuss two representative examples that illustrate the contents of Theorem \ref{T:maintheorem}.
\subsection{Functions of one variable}
Let us consider the polynomial $p=(1-z_1)(1-z_2)$ and investigate its properties as an element of $\Dveca$. Its zero set in $\partial \DD^2$ is rather large: it can be represented as 
\[\mathcal{Z}(p)\cap \partial \DD^2=(\{1\}\times \overline{\DD})\cup (\overline{\DD})\times \{1\}).\]

Nevertheless, $p$ is cyclic in all $\Dveca$ having $\max\{\alpha_1,\alpha_2\}\leq 1$. This follows from the 
cyclicity of $P=1-z$ in the classical Dirichlet space on the disk originally established by Brown and Shields \cite{BS84}. 

\begin{theorem}[Brown and Shields, 1984]\label{BrownShields}
If $P$ is a polynomial with no zeros in $\DD$, then $P$ is cyclic in $D_{\alpha}$ for $\alpha\leq 1$.
\end{theorem}

We give an alternate proof of this theorem, using dilations of $1/P$ and the following integral estimate of Forelli and Rudin (see \cite[Theorem 1.7]{HKZBook}): for $a\in (-1, \infty)$ and $b\in (-\infty,\infty)$,
\begin{equation}
\int_{\DD}\frac{(1-|z|)^{a}}{|1-\overline{w}z|^{2+a+b}}dA(z)\asymp\left\{ \begin{array}{cc}1, &b<0\\
-\log(1-|w|^2), & b=0\\
(1-|w|^2)^{-b}, &b>0
\end{array}
\right.
\label{integralasymptotics}
\end{equation}
as $|w|\to 1^{-}$.  This estimate will be useful later on in the paper
as well.

The {\it radial dilation} of a function $f\colon\DD\to\CC$ is defined for $r\in (0,1)$ by $f_r(z)=f(rz)$. 

\begin{lemma}\label{onevardilation}
Let $P$ be a polynomial with no zeros in $\DD$.
Then $P/P_r\to1$  weakly   in $D_1$ as $r\to1^-$.
If $\alpha<1$, then $P/P_r\to1$ in norm in $D_\alpha$ as $r\to1^-$.
\end{lemma}

\begin{proof}
We consider first the case $\alpha=1$. It suffices to establish that $P/P_r$ is bounded in the Dirichlet norm. Indeed, if this holds, then the family $P/P_r$ is relatively weakly compact, and since $P/P_r\to 1$ pointwise, weak convergence follows. (See \cite{BS84} for a comprehensive discussion of convergence concepts.) Every $P$ can be factored into linear factors, and it is not hard to see that every such factor $Q$ satisfies the estimate $|Q(z)/Q_r(z)| \leq 2$ for $z\in \DD$. Thus, in showing that $P/P_r$ is bounded in the Dirichlet norm, it suffices to treat the case where $\deg P=1$. Also, we may as well as assume that the zero of $P$ lies on $\TT$, the other case (when the zero lies outside $\overline{\DD}$) being obvious. Finally, by rotation invariance, we can suppose that $P(z)=1-z$.
Computing the Dirichlet integral of $P/P_r$, and invoking the estimate \eqref{integralasymptotics} with $a=0$ and $b=2$, we find that
\[\|P/P_r\|_{D_1}^2\asymp1+ \int_{\DD}\frac{(1-r)^2}{|1-rz|^4}dA(z)\asymp1+\frac{(1-r)^2}{(1-r^2)^2}\asymp 1,\]
and the proof for $\alpha=1$ is complete.

The case $\alpha<1$ may be treated in a similar way. Alternatively one can remark that, since the inclusion  $D_1\hookrightarrow D_\alpha$ is a compact linear map, weak convergence in $D_1$ carries over to norm convergence in $D_\alpha$.
\end{proof}

\begin{proof}[Proof of Theorem~\ref{BrownShields}]
Recall that $[P]$ is the smallest closed invariant subspace of the Dirichlet space $D_\alpha$ that contains a given $P$. By definition, it contains all functions of the form $q\cdot P$, where $q$ is a polynomial. It also contains all functions of the form 
$g\cdot P$, where $g$ is holomorphic on a neighborhood of $\overline{\DD}$, because the Taylor polynomials of $g$ converge to $g$ in the multiplier norm of $D_\alpha$. We therefore have $P/P_r\in [P]$ for each $r\in (0,1)$. We have shown that $P/P_r\to 1$ weakly in $D_\alpha$ for all $\alpha\le1$. As $[P]$ is weakly closed in $D_\alpha$, it follows that $1\in [P]$, and hence $[P]=D_\alpha$, as desired.
\end{proof}

We return to the two-variable setting and the product function $P=(1-z_1)(1-z_2)$. Since the $\Dveca$-norm restricted to functions depending on the first variable only coincides with the norm in $D_{\alpha_1}$, and the $\Dveca$-norm reduces to the norm in $D_{\alpha_2}$ for function of $z_2$ only, Theorem~\ref{BrownShields} implies that at least one of the factors of $p$ is cyclic precisely when $\min \{\alpha_1, \alpha_2\}\leq 1$. They are both cyclic if $\max\{\alpha_1, \alpha_2\}\leq 1$, and since the product of cyclic multipliers is itself cyclic, the result follows.
\subsection{A function vanishing along a curve}
To illustrate what happens when the zero set is infinite, let us consider the polynomial $p(z_1, z_2)=1-z_1z_2$, which does not 
vanish in the bidisk or on $\partial \DD^2\setminus \TT^2$, and has \[\mathcal{Z}(p)\cap \TT^2=\{(e^{it}, e^{-it})\},\]
a curve in the distinguished boundary.

The same kind of reasoning as in \cite[Section 3]{BCLSS13II} reveals that $p$ is cyclic in $\Dveca$ precisely when $\alpha_1+\alpha_2\leq 1$. We first recall that an equivalent criterion  for a function $f\in \Dveca$ to be cyclic is the existence of a sequence $(q_n)$ of polynomials in two variables such that 
\[\|p\cdot q_n-1\|_{\veca}\to 0 \quad \textrm{as}\quad n\to \infty,\]
for then the cyclic element $1\in [f]_{\veca}$.
By orthogonality, the expression $\|p\cdot q_n-1\|_{\veca}$ is minimized 
by taking $(q_n)$ to be polynomials in $z_1z_2$. Next, it 
is elementary to see that 
\[\|f\|_{\veca}=\|F\|_{D_{(\alpha_1+\alpha_2)}}\] 
for functions $f\in \Dveca$ of the form $f(z_1,z_2)=F(z_1\cdot z_2)$, where $F\colon \DD\to \CC$ is a function on the unit disk. But the polynomial $P=1-z$ is cyclic in $D_{\alpha}$ if and only if $\alpha \leq 1$. Hence $p=1-z_1z_2$ is cyclic precisely when $\alpha_1+\alpha_2\leq1$. 
\subsection{A function vanishing at a single point}
We now consider a polynomial that does depend on both variables, but whose 
zero set in the distinguished boundary is minimal.
\begin{example} 
The polynomial $p(z_1,z_2)=2-z_1-z_2$ is cyclic for $\Dveca$ whenever $\min\{\alpha_1,\alpha_2\}\le1$.
\end{example}
We note that $\mathcal{Z}(p) \cap \TT^2=\{(1,1)\}$,  so that $p$ can certainly not be cyclic in $\Dveca$ when $\min\{\alpha_1,\alpha_2\}>1$.

This example will follow from Section \ref{sec:finite}, where we
rehash an argument from \cite{BKKLSS14}.  In Appendix A, we offer an
elementary proof of this fact which may be of independent
interest. 

\section{Polynomials with finite zero sets} \label{sec:finite}
We now turn to the proof of Theorem \ref{T:maintheorem}.  The case of
a polynomial that does not vanish on the closed bidisk is trivial, so we
exclude it.  Our first step is to address finitely many zeros on $\TT^2$.

\begin{theorem} Let $p\in \CC[z_1,z_2]$ have no zeros in the open
  bidisk, and finitely many zeros on $\TT^2$. Then $p$ is cyclic in
  $\Dveca$ for $\min\{\alpha_1, \alpha_2\}\leq 1$. 
\end{theorem}

The proof is a very slight modification of a corresponding argument in
 Section~3 of \cite{BKKLSS14}. Without loss of generality, assume
$\alpha_1\leq \alpha_2$, so that in particular $\alpha_1\leq 1$. The
basic idea is to compare our polynomial to a product of factors of the
form $\zeta-z_1$, which are known to be cyclic.

Write 
\[\mathcal{Z}(p)\cap \TT^2=\{\zeta^1, \zeta^2, \ldots, \zeta^N\}\]
for some $N\in \NN$; each $\zeta^k=(\zeta_1^k, \zeta_2^k)$ with
$|\zeta^k_1|=|\zeta^k_2|=1$.

We now recall {\it \L ojasiewicz's inequality}, a classical result in real algebraic geometry \cite{KPBook}:
if $f$ is a real analytic function on an open set $U\subset \RR^d$, and $E\subset U$ is compact, then there exist a constant $C>0$ and a number $q\in \NN$ such that 
\begin{equation}
|f(x)|\geq C\cdot \mathrm{dist}(x, \mathcal{Z}(f))^q, \quad x\in E.  
\label{E:loj}
\end{equation}
Apply \eqref{E:loj} to the function $|p|^2$, which is real-analytic on
$\CC^2$, and the compact set $E = \TT^2$ to see there exist $C>0,
q\in \mathbb{N}$ such that
\[
\begin{aligned}
|p(z)|^2 \geq C \text{dist}(z, \mathcal{Z}(p))^{2q} &\geq C_1
\prod_{k=1}^n(|z_1-\zeta^k_1|^2+|z_2-\zeta^k_2|^2)^{q}\\
&\geq C_1 \prod_{k=1}^{N}|z_1-\zeta^k_1|^{2q}, \quad z\in \TT^2.
\end{aligned}
\]
Thus, 
\[Q(z_1,z_2)=\frac{\prod_{k=1}^N(z_1-\zeta^k_1)^q}{p(z_1,z_2)}\]
is bounded on $\TT^2$ and if we increase $q$ we can make this function
as smooth as we like on $\TT^2$.  

But then $Q\in \mathfrak{D}_{\veca}$ because of rapid decay of Fourier
coefficients of $Q$, so that in turn
$g:=pQ=\prod_{k=1}^N(z_1-\zeta^k_1)^q\in p\Dveca$. Now $p$ is a
multiplier, hence $p\Dveca=[p]_{\veca}$, and $g$ is cyclic in $\Dveca$
for $\alpha_1\leq 1$, since it is a product of cyclic
multipliers. Hence $[p]_{\veca}$ contains a cyclic element for
$\Dveca$.

\section{Infinite zero sets} 
This section is devoted to proving that a polynomial with no zeros in
$\DD^2$ and infinitely many zeros in the $2$-torus  is cyclic in $\Dveca$ exactly when
$\al_1+\al_2 \leq 1$.  The proof of cyclicity when $\al_1+\al_2\leq 1$
and the proof of non-cyclicity when $\al_1+\al_2>1$ have the
interesting feature that they in some sense reduce the problem to the
model polynomial $1-z_1z_2$.

\subsection{Cyclicity via radial dilations} 
By the previous section we only need to address
polynomials with infinite zero set on $\TT^2$ (necessarily forming a
curve) but the proof below does not use this in an essential way.

\begin{theorem}\label{T:raddiltwovars} Suppose
  $p\in \CC[z_1,z_2]$ is irreducible, has no zeros in the bidisk, and is
  not a polynomial in one variable only.  Then $p$ is cyclic in $\Dveca$ whenever
  $\alpha_1+\alpha_2\leq 1$. 
\end{theorem}

We will use $1-z_1z_2$ as a ``model polynomial," analogous to $1-z$ in the one-variable setting.
The crux of the proof of Theorem \ref{T:raddiltwovars} is to establish the following analog of Lemma \ref{onevardilation}; see also \cite{NGN70, G95}.
\begin{lemma}\label{twovardilation} Let $\alpha_1+\alpha_2\leq 1$. Then 
\begin{equation}
F_r(z_1,z_2)=\frac{p(z_1 ,z_2)}{p(rz_1, z_2)}\longrightarrow 1
\label{dilationconvergence}
\end{equation} 
weakly in $\Dveca$ as $r\to 1^{-}$.
\end{lemma}
\begin{lemma}\label{sides}
Suppose a polynomial $p\in \CC[z_1,z_2]$ is irreducible, has no zeros on $\DD^2$, and is not a polynomial in one variable only. Then $p$ has no zeros on $(\DD\times \TT)\cup (\TT\times \DD)$.
\end{lemma}
\begin{proof}
 For fixed $a \in
\DD$, $z \mapsto p(z,a)$ has no zeros in $\DD$.  By Hurwitz's theorem,
if we send $a \to \TT$, then $z\mapsto p(z,a)$ either has no zeros in $\DD$,
or is identically zero.  If it is identically zero, then $z_2-a$
divides $p(z_1,z_2)$, contrary to our assumptions.  Thus, for each $a
\in \TT$, $z \mapsto p(z,a)$ has no zeros in $\DD$.  By symmetry, 
$p$ has no zeros on $\DD\times \TT$ either.
\end{proof}
\begin{proof}[Proof of Theorem \ref{T:raddiltwovars}] 

Assuming Lemma \ref{twovardilation}, cyclicity of $p$ in $\Dveca$ with $\alpha_1+\alpha_2\leq 1$ follows as before. To see this, set $p_r(z_1,z_2)=p(rz_1, z_2)$. By Lemma~\ref{sides}, the polynomial $p_r$ does not vanish on the closed bidisk for $r<1$. Thus each $1/p_r$ extends holomorphically past the closed bidisk, and hence is a multiplier, which in turn implies $F_r\in [p]_{\veca}$ for each $r<1$. Thus $1\in [p]$, as a limit of $F_r$, and we are done.
\end{proof}

Before we give the proof of Lemma \ref{twovardilation}, we make some preliminary remarks, and a few reductions. Suppose that the given polynomial $p$ has bidegree $(m,n)$---that is, degree $m$ in the variable $z_1$ and degree $n$ in the variable $z_2$. First of all, we can view $p$ as a polynomial in the variable $z_1$, 
\[p(z_1,z_2) = A_m(z_2) z_1^m+ \cdots + A_1(z_2) z_1 + A_0(z_2),\]
with coefficients $A_k$ that are polynomials in $z_2$. 
Since $p$ does not vanish on the bidisk we infer that $A_0$ does not
vanish on the disk.  In fact, $A_0$ does not vanish on the unit
circle, for if it did then $p$ would vanish on $\{0\}\times \TT$.
Such a zero is ruled out by the assumptions that $p$ is irreducible and
depends on both variables via Lemma~\ref{sides}.

For each fixed $z_2$ where $A_m(z_2) \ne 0$, we can therefore factor
$p$ into linear factors in the variable $z_1$,
\begin{equation}
p(z_1,z_2)=A_0(z_2)(1-z_1 h_1(z_2))\cdots (1-z_1 h_m(z_2)).
\label{zonefactorization}
\end{equation}
Before we proceed, we need to discuss the nature of the functions that appear in the 
right-hand side: the $h_j$ require a particularly careful treatment.

The function $A_0=A_0(z_2)$ is a polynomial in one complex
variable with no zeros in the closed unit disk, and hence $A_0$  is a multiplier and cyclic 
in every space $\Dveca$.  We can of course reverse the roles of $z_1$
and $z_2$ in the factorization and write for each fixed $z_1$ (outside
a finite set) 
\begin{equation}
p(z_1,z_2)=B_0(z_1)(1-z_2 g_1(z_1))\cdots (1-z_2 g_n(z_1)),
\label{ztwofactorization}
\end{equation}
and again obtain a non-vanishing polynomial $B_0=B_0(z_1)$ furnishing
a cyclic multiplier.  Since the product $\phi f$ of a function
$f\in \Dveca$ and a multiplier $\phi \in M(\Dveca)$ is cyclic if and
only if both factors are cyclic \cite[Proposition 8]{BS84}, it is
enough to establish cyclicity of the factor that depends on both
variables.  Namely, we can drop the factors $A_0$ or $B_0$.

In summary, we may assume without loss of generality that, outside of a discrete set, 
the polynomial $p$ is locally of the form
\begin{equation}
p(z_1,z_2)=(1-z_1 h_1(z_2))\cdots (1-z_1 h_m(z_2)).
\label{linearfacts}
\end{equation}
The functions $h_j=h_j(z_2)$, $j=1,\ldots, m$, are more problematic as they are
no longer polynomials (or even single-valued), but can be represented
as branches of algebraic functions. In particular, the factors $1-h_j
z_1$ are in general not elements of $\Dveca$.  

\begin{example}\label{E: branchpoints}
  For each $a\in [0,1)$, the zero set of the irreducible polynomial
  $p=1-az^2_1-az_2+z^2_1z_2$ 
 meets $\overline{\mathbb{D}^2}$ in a
  curve in $\TT^2$. For this polynomial, we have the local
  factorization
\[p(z_1,z_2)=(1-az_2)\left(1-\left(\frac{a-z_2}{1-az_2}\right)^{1/2}z_1\right)
\left(1+\left(\frac{a-z_2}{1-az_2}\right)^{1/2}z_1\right),\]
and hence $h_j$, $j=1,2$, have a branch point at $a\in \DD$.
\end{example} To get around the difficulty posed by the nature
of the functions $h_k$, we let $D=D_p$ be a simply connected subdomain
of $\DD$ such that each $h_j$ extends holomorphically to $D$, and
$|\DD| = |D|$ (here $|\cdot |$ denotes 2-dimensional Lebesgue
measure).  More precisely, first notice that we allow some $h_j(z_2)$
to equal zero but this necessarily occurs at the finitely many $z_2$
where $A_m(z_2) = 0$.  We can think of $(\infty,z_2)$ as a root of $p$
in such cases.  Now define
\[
\begin{aligned}
S =& \left\{z_2 \in \CC\colon p(\cdot, z_2), \frac{\partial p}{\partial
  z_1}(\cdot, z_2) \text{ have a common root} \right\} \\
&\cup \{z_2 \in \CC\colon A_{m}(z_2) = A_{m-1}(z_2) = 0\}
\label{singset}
\end{aligned}
\]
This set is necessarily finite because $p$ is irreducible.  The first
set in the union consists of the points where $p(\cdot,z_2)$ has a
repeated root and the second set consists of the points where
$p(\cdot,z_2)$ has a repeated root ``at $\infty$.''  On
$\CC \setminus S$, there are distinct $h_1(z_2),\dots, h_m(z_2)$ which
can be defined locally in an analytic fashion and satisfy
\eqref{linearfacts}.  Given $z \in S$ form the ray
$R_z = \{tz : t\geq 1\}$, let $\mathcal{R} = \cup_{z \in S} R_z$, and
finally define
\[
D := \DD \setminus \mathcal{R}.
\]
On $D$, the $h_k$ can be analytically continued to single-valued
analytic functions satisfying \eqref{linearfacts}.  Notice that across
the boundary slits of $D$ the $h_k$ can be analytically continued
necessarily to some $h_j$.  Near a point of $S$, two of the $h_k$ tend
to the same value (the case of $A_m(z_2)=A_{m-1}(z_2)=0$ means that
two of the $h_k$ tend to zero at $z_2$ as in Example
\ref{E: branchpoints}).  The $h_j$ will extend analytically to any
point of $\TT\setminus S$.

We may assume $0\in D$; if this is not the case, we can
replace $p(z_1,z_2)$ by $(1-az_2)^np(z_1, \varphi(z_2))$, where
$\varphi=\varphi(z)$ is a M\"obius transformation of the unit
disk. Cyclicity of one of these two functions implies cyclicity of the
other. Cf. Example \ref{E: branchpoints}, where $1+z^2_1z_2$ ($a=0$) is
transformed into $1-az^2_1-az_2+z^2_1z_2$ ($a>0$) in precisely this way.
\begin{lemma}\label{L: hfunctlemma}
Let $h_k\colon D\rightarrow \CC$ be as above. Then $|h_k(z)|<1$ in
$D$, $h_k$ has bounded multiplicity, and $(1-|h_k(z)|^2)/(1-|z|^2)\geq C >0$.
\end{lemma}
\begin{proof}
If $|h_k(z)|\geq 1$ for some $z\in D$ then the polynomial $p$ would vanish inside $\DD^2$, contradicting our assumptions. Hence $|h_k(z)|<1$ for all $z\in D$, and for every $k$. Furthermore, if for some $a\in \DD$ it were the case that $h_k(z_2)=a$ for more than $n$ values of $z_2$, then $p(a^{-1},z_2)$ 
would have more than $n$ zeros, which is impossible since $p$ is irreducible of bidegree $(m,n)$.

Next, put
$$u(\lambda):= \max \{ |h_k(\lambda)|\colon\ k=1,\ldots,n\},\quad
\la\in \overline{\DD}.$$
By the discussion before the lemma, this is well-defined in a
neighborhood of $\overline{\DD}$ minus $S$ since each $h_k$
analytically continues across the boundary slits of $D$ and to
$\TT\setminus S$.  Thus $u$ is subharmonic on
$\overline{\DD}\setminus S$ and extends to be subharmonic on
$\overline{\DD}$ because the points of $S$ will be removable
singularities. Clearly $u(\lambda)<1$, $\lambda\in \DD$.  

According to the Hopf lemma for subharmonic functions (see
\cite[Proposition 12.2]{ForStenBook}) there is $C>0$ such
that $$u(\la) - 1\leq C( |\la|-1),\quad \la \in \DD.$$ The last
assertion of the lemma follows.
\end{proof}

Finally, we shall need the following lemma.
\begin{lemma}[Analytic maps of bounded multiplicity]\label{L:changevar}
Let $D\subset \DD$ be a domain, and suppose $\phi:D\to \DD$ is an analytic map of multiplicity at most 
$M$. Then, for any non-negative function $g\in L^2(\DD)$,  we have
\begin{equation}
\label{multiplicityest}\int_D (g\circ \phi(z)) |\phi'(z)|^2  dA(z)\leq M \int_\DD g(w)dA(w).
\end{equation}
\end{lemma}
\begin{proof}
See \cite[Lemma 6.2.2]{EKMRBook}.
\end{proof}
Since the $h_k$ appearing in Lemma~\ref{L: hfunctlemma} are of bounded multiplicity, we apply 
Lemma \ref{L:changevar} to $g=h_k$, and will do so frequently in what follows.

\medskip
To begin our proof of Lemma \ref{twovardilation}, we set, for
$r\in (0,1)$,
$$F_r(z_1,z_2) = 
q_r(z_1,h_1(z_2)) \cdots q_r(z_1, h_m(z_2)),$$ where 
\[q_r(z_1,z_2) = \frac{1- z_1z_2}{1-rz_1z_2}.\]
Since pointwise convergence of $1/p_r\to 1/p$ as $r\to 1^{-}$ clearly
holds, it suffices, as in the proof of Lemma \ref{onevardilation}, to
show that the $\Dveca$-norm of $F_r$ remains bounded as $r \to 1^{-}$.

\begin{lemma} \label{model}
\[
\int_{\DD^2} \frac{1-r}{|1-rz_1z_2|^4} dA(z_1)dA(z_2) 
\]
is bounded by a constant for $r \in (0,1)$.
\end{lemma}
\begin{proof}
By \eqref{integralasymptotics} the integral is bounded by a constant times
\[
\int_{\DD} \frac{1-r}{(1-|rz_1|^2)^2} dA(z_1)
\]
and the above integral is equal to a constant times $1/(1+r)$.
\end{proof}

We record some qualitative features of the derivatives of $F_r$.
\begin{lemma} \label{qual} The function
\[
q_r(z_1,z_2) = 1 + (r-1)\frac{z_1z_2}{1-rz_1 z_2}
\]
is bounded in $\DD^2$ independent of $r$.  For fixed $k$, the
derivatives $\partial_{z_1}^k q_r$ and
$\partial^{k-1}_{z_1} \partial_{z_2}q_r$ are of the form
\[
(1-r)\frac{G(z_1,z_2,r)}{(1-rz_1z_2)^{k+1}}
\]
for some polynomial $G$. 
Consequently, the derivatives 
\[
\text{(a) } \partial^k_{z_1} F_r(z_1,z_2) \text{ and
} \text{(b) } \partial^{k}_{z_1} \partial_{z_2} F_r(z_1,z_2)
\]
can be written as a finite sum of terms of the form $B(z;r) P(z;r) H(z;r)$
where 
\begin{itemize}
\item $B$ is bounded on $\DD\times D\times (0,1)$
\item for some integer $l$, $0\leq l\leq k$, $P(z;r)$ equals the product of  $(1-r)^l$ and
  $k+l$ terms of the form $\frac{1}{1-rz_1h_j(z_2)}$  for multiple
choices of $j$,
\item in case (a), $l\ne 0$ and $H$ is unnecessary and in case (b) $H$
  is of the form $(1-r)\frac{h_j'(z_2)}{(1-rz_1h_j(z_2))^2}$ for some
  $j$.
\end{itemize}

Assuming $1\geq \al_1+\al_2$ and $1>\al_2$, and $2k-1 > \al_1$, we have the estimate
\begin{multline} \label{mainest}
|\partial_{z_1}^k \partial_{z_2} F_r(z)|^2 \\  \leq
\frac{C(1-r)}{(1-|z_1|^2)^{2k-1-\al_1} (1-|z_2|^2)^{1-\al_2}}
  \sum_{j=1}^{m} \frac{|h_j'(z)|^2}{|1-rz_1h_j(z_2)|^4} 
\end{multline}
valid in $\DD\times D$ for some constant $C>0$.

\end{lemma}

\begin{proof}
The formula for $q_r$ shows $|q_r|\leq 2$.  The other formulas are
calculus exercises.

To get the final estimate let $j_1,\dots, j_{k+l} \in \{1,\dots, m\}$.
Then, $|P|^2$ is of the form
\[
|P(z;r)|^2 = \frac{(1-r)^{2l}}{\prod_{i=1}^{k+l} |(1-rz_1 h_{j_i}(z_2))|^2}.
\]
Terms of the form $|1-rz_1h_j(z_2)|$ appear in the denominator $2k+2l$
times.  We use the estimate $|1-rz_1h_j(z_2)| \geq (1-r)$ for $2l+1$
of the terms.  Evidently, $|1-rz_1h_j| \geq C(1-|z_1|^2)$ and by Lemma
\ref{L: hfunctlemma} we have $|1-rz_1 h_j| \geq C(1-|z_2|^2)$.  Using
either of these two estimates for the remaining $2k-1$ terms in the
denominator of $|P|^2$ gives the bound
\begin{equation} \label{P2}
|P|^2 \leq \frac{C}{1-r} \min\left\{ \frac{1}{(1-|z_1|^2)^{2k-1}},
  \frac{1}{(1-|z_2|^2)^{2k-1}}\right\}
\end{equation}
The assumptions on the parameters mean $2k-1 \leq
(2k-1-\al_1) + (1-\al_2)$ with both summands non-negative.  This
immediately yields 
\[
|P|^2 \leq \frac{C}{1-r} \frac{1}{(1-|z_1|^2)^{2k-1-\al_1}
  (1-|z_2|^2)^{1-\al_2}}
\]
for some constant $C>0$.  Combining this with the definition of $H$,
we get 
\[
|BPH|^2 \leq C\frac{(1-r)|h_j'(z_2)|^2}{(1-|z_1|^2)^{2k-1-\al_1}
  (1-|z_2|^2)^{1-\al_2} |1-rz_1 h_j(z_2)|^4}.
\]
Since $|\partial^k_{z_1} \partial_{z_2}F_r|^2$ can be bounded by a
finite sum of such terms the main estimate \eqref{mainest} holds.

\end{proof}

In view of our construction of the domain $D$, the seminorm
$\Dveca(F_r)$ is not affected if we restrict integration to
$\DD\times D$. Note that the integral formula for $\Dveca(F_r)$ is
valid if both $\alpha_1<2$ and $\alpha_2<2$.  Estimate \eqref{mainest}
can be used with $k=1$ if we assume $\al_1,\al_2<1$.  Certainly, since
$\al_1+\al_2 \leq 1$ we can assume without loss of generality that
$\al_2<1$.  Later on we will adjust the argument in case $\al_1 \geq
1$. 

The problematic term in $\Dveca(F_r)$
is the integral
\begin{equation}\label{eq}
\int_{\DD\times D}|\partial_{z_1} \partial_{z_2}
F_r(z_1,z_2)|^2(1-|z_1|^2)^{1-\al_1} (1-|z_2|^2)^{1-\al_2} dA(z_1) dA(z_2).
\end{equation}

By \eqref{mainest} with $k=1$, this is bounded by
\begin{multline*}
C(1-r) \sum_{j=1}^{m} \int_{\DD\times D}
\frac{|h_j'(z_2)|^2}{|1-rz_1h_j(z_2)|^4} dA(z_1) dA(z_2)\\ \leq C(1-r) m
\int_{\DD^2} \frac{1}{|1-rz_1z_2|^4} dA(z_1)dA(z_2)
\end{multline*}
by Lemma \ref{L:changevar}.  The integral on the right is bounded by a constant by
Lemma \ref{model}.  This proves $\Dveca(F_r)$ is bounded independent of
$r\in (0,1)$ when $\al_1,\al_2<1$.  

Suppose now that one of the parameters, say $\al_1$, satisfies
$\alpha_1\geq 1$. We then use that $f\in \Dveca$ if and only if
$\partial_{z_1}f\in \mathfrak{D}_{\alpha_1-2, \alpha_2}$ and
$f(0,\cdot) \in D_{\al_2}$ to switch from $F_r$ to
$\partial^{N-1}_{z_1} F_r$, where $N$ is chosen so that
$2N-1 >\alpha_1$. To bound the $\Dveca$-norm of $F_r$, it is now
enough to bound the norms $\|\partial_{z_1}^l F_r(0,\cdot)\|_{D_{\al_2}}$
for $l\leq N-1$, the one-variable norms
$\|\partial^{N-1}_{z_1} \partial_{z_2} F_r(0,\cdot)\|_{D_{\al_2}}$ and 
$\|\partial^{N}_{z_1} F_r(\cdot,0)\|_{D_{\alpha_1-2N}}$, and the double integral
\begin{equation} \label{hardterm}
\int_{\DD\times D}|\partial^{N}_{z_1} \partial_{z_2} F_r(z_1,z_2)|^2
(1-|z_1|^2)^{2N - 1 -\alpha_1}(1-|z_2|^2)^{1-\al_2} dA(z_1) dA(z_2).
\end{equation}

By case (a) of Lemma \ref{qual}, $\partial^l_{z_1}F_r(0,\cdot)$ is bounded; this
takes care of the contribution $\|\partial^l_{z_1} F_r(0,\cdot)\|_{D_{\al_2}}$.  By
case (b) of Lemma \ref{qual}, $|\partial_{z_1}^{N-1}\partial_{z_2}
F_r(0,z_2)|^2$ can be controlled by terms of the form
$|h_j'(z_2)|^2$.  But, since $\al_2 <0$
\[
\int_{D} |h_j'(z_2)|^2 (1-|z_2|^2)^{1-\al_2} dA(z_2) \leq 
\int_{D} |h_j'(z_2)|^2 dA(z_2) \leq C\int_{\DD} |z_2|^2 dA(z_2)
\]
which is finite.  By case (a) of Lemma \ref{qual}, to control
$|\partial_{z_1}^NF_r(z_1,0)|^2$ it suffices to control $|P(z_1,0;r)|^2$,
which by \eqref{P2} 
can be bounded by 
\[
\frac{1}{(1-u(0))^{2N}};\]
where $u(0) = \max_{j} |h_j(0)|$. Now this is enough to show that the integral $\int_{\DD}
|\partial_{z_1}^N F_r(z_1,0)|^2 dA_{\al_1-2N}(z_1)$ is bounded since
$\al_1-2N<0$.

Finally, \eqref{hardterm} can be bounded using \eqref{mainest} with
$k=N$ in the same way we bounded \eqref{eq}.  This concludes the proof
of Lemma \ref{twovardilation} and consequently Theorem
\ref{T:raddiltwovars}.

\subsection{Non-cyclicity for infinite zero sets}
We now complete the proof of our main theorem by proving that an
irreducible polynomial that has no zeros in $\DD^2$ but vanishes along
a curve in $\TT^2$ fails to be cyclic in certain $\Dveca$. In the sequel we lose no generality assuming that $\alpha_2\leq 1$.
\begin{theorem}\label{T: infinitezerosthm}
  Let $p\in \CC[z_1,z_2]$ be an irreducible polynomial depending on
  both variables with infinitely many
  zeros in $\TT^2$. Then $p$ is not cyclic in $\Dveca$ whenever
  $\alpha_1+\alpha_2>1$.
\end{theorem}
It is clear that the polynomial $p$ is not cyclic if it vanishes in the open bidisk, so we exclude that case in what follows. Then the idea of the proof is again to argue by reduction to the polynomial 
$1-z_1z_2$: we show that if the polynomial $p$ is cyclic, then so is $1-z_1z_2$, leading to a contradiction 
when $\alpha_1+\alpha_2>1$.

To this end, let $p$  be irreducible, depending on
  both variables with infinitely many
  zeros in $\TT^2$ and no zeros on $\DD^2$. Then, as is
pointed out in \cite[Section 2.2]{BKKLSS14}, we must have
\begin{equation}\label{eq:zerf} 
\mathcal{Z}(p)\subset \TT^2 \cup (\EE \times \DD) \cup (\DD\times \EE),
\end{equation}
where $\EE=\CC\setminus \overline{\DD}$. (This is not so in the case
of {\it finitely many} zeros: $2-z_1-z_2$ vanishes at
$(4,-2)\in \EE^2$ for example.)  This follows from the fact that the
irreducible polynomial $p$ will have infinitely many zeros in common
with $z_1^mz_2^n\overline{p(1/\bar{z}_1, 1/\bar{z}_2)}$ which implies
these polynomials are constant multiples of one another by B\'ezout's
theorem.  Thus, $p$ has no zeros in $\DD^2\cup\EE^2$.  By Lemma
\ref{sides}, $p$ has no zeros in $(\TT\times\DD)\cup(\DD\times \TT)$
and by reflection no zeros in $(\TT\times \EE)\cup(\EE\times \TT)$.
This leaves only the set \eqref{eq:zerf}. 

As in the previous section, we write
$$p(z_1,z_2) = A_0(z_2) (1- h_1(z_2) z_1)\cdots (1-h_m(z_2)z_1)$$
and 
$$p(z_1,z_2) = B_0(z_2) (1- g_1(z_1) z_2)\cdots (1-g_n(z_1)z_2).$$
Note that for every $j=1,\ldots, n$, there is a $k=k(j)$ such that 
\begin{equation}\label{eq:alpbet}
h_k(g_j(z_2^{-1})^{-1})=z_2.
\end{equation}
This holds because each root of $p$ can be written in the form
$(z^{-1}_1,g_j(z_1^{-1})^{-1}) = (h_k(z_2)^{-1}, z_2)$ (outside a finite
set of values of $z_1$ and $z_2$).   

Now put $b_j(z_2):=g_j(z_2^{-1})^{-1}$, and note that \eqref{eq:zerf} implies that $b_j(z_2) 
\in \DD$ for $z_2\in \DD\setminus S$, where $S$ is the set defined in the previous subsection.

 Applying Puiseux's theorem in the form of Lemma~\ref{rootglue} in Appendix B to 
\[F(z_1,z_2):=(z_1-b_1(z_2))\cdots (z_1-b_n(z_2))\] we obtain bounds on the derivatives of $b_j$ on a neighborhood of  the closed unit disk. Namely, for $a\in S$,
\begin{equation}
|b_j'(z)|\leq O\left(\frac{1}{|z-a|^{1-1/n}}\right).
\label{puiseuxest}
\end{equation}

From the equalities \eqref{eq:alpbet} we deduce that 
$$p(z_1, b_1(z_2) )\cdots p(z_1, b_n(z_2) ) = (1-z_1z_2)^n
G(z_1,z_2), $$
 where $G$ is a function holomorphic on a neighborhood of $\overline{\DD^2}$.
\begin{lemma}[Symmetrization lemma]\label{L:symmetrization}
Let $(P_{\nu})_{\nu=1}^{\infty}$ be a sequence of polynomials in two complex variables, set 
$q_{\nu}=p\cdot P_{\nu}$, and suppose $\|q_{\nu}-1\|_{\veca}\to 0$ as $\nu\to \infty$.

Then, for any $\epsilon>0$, there are $\nu_1,\ldots,\nu_n$ such that for any permutation $\sigma\colon 
\{1, \ldots,n\}\to \{1,\ldots, n\}$, we have the seminorm estimate \[\Dveca\left(q_{\nu_1}(z_1, b_{\sigma(1)}(z_2))\cdots q_{\nu_n}(z_1, b_{\sigma(n)}(z_2))\right)<\epsilon.\]
\end{lemma}

Assuming Lemma \ref{L:symmetrization}, we now present the proof of Theorem \ref{T: infinitezerosthm}.
\begin{proof}[Proof of Theorem \ref{T: infinitezerosthm}]
Suppose $p$ is cyclic. Then there is a sequence of polynomials $(P_\nu)_{\nu=1}^\infty$ such that $\|p \cdot P_\nu -1 \|_{\vec \alpha}\to 1$. Put $q_\nu = p \cdot P_\nu$ and let $\mathfrak{S}_n$ denote the symmetric group of order $n$. Define
$$F_\mu(z_1,z_2):=\frac{1}{n!}\sum_{\sigma\in \mathfrak{S}_{n}} q_{\mu_1}(z_1, b_{\sigma(1)}(z_2))\cdots q_{\mu_n}(z_1, b_{\sigma(n)}(z_2)).$$ By Lemma \ref{L:symmetrization}, $F_\mu$ can be brought arbitrarily close to $1$ in $\Dveca$ through a suitable choice of $\mu$. 

We note that $F_\mu(z_1,z_2) = (1-z_1z_2)^n Q_\mu(z_1,z_2)$, where
$Q_\nu$ is holomorphic on a neighborhood of $\overline{\DD^2}$ and
hence a multiplier. Now $Q_{\mu}$ can in turn be approximated in
multiplier norm by polynomials, and this shows that $(1-z_1z_2)^n$ is
a cyclic function. Then also $1-z_1z_2$ is cyclic since a product of
multipliers is cyclic if and only if each factor is cyclic. But as we have seen, $1-z_1z_2$  is only  cyclic
when $\alpha_1+\alpha_2\leq 1$, and the Theorem follows.
\end{proof}
\bigskip

We are left with the proof of the lemma. This will require the following one-variable estimate.
\begin{lemma}\label{L: quotientderivative}
Let $a<1$ and $\beta\leq 1$. Then there 
is a constant $C$ depending only on $a$ and $\beta$ such that for any $g\in \mathrm{Hol}(\DD)$,
\begin{equation}
\label{te}\int_\DD \left|\frac{g(z)}{(z-1)^a}\right|^2 dA_{\beta}(z) 
\leq C\left(|g(0)|^2 + \int_\DD |g'(z)|^2 dA_{\beta}(z)\right).
\end{equation}
\end{lemma}
\begin{proof}
We assume that $g(z)=\sum_{n\geq 0} a_n z^n$ belongs 
to the Dirichlet-type space $D_\beta$; otherwise, the inequality is trivial. 
We may also assume that $a_0=0$. Applying the Cauchy-Schwarz inequality to 
the expression $|\sum_{n}a_nz^n|^2$, we have 
\begin{multline*}
\int_\DD \frac{|g(z)|^2 }{|1-z|^{2a}} (1-|z|^2)^{1-\beta} dA(z) 
\leq \left(\sum_{n=1}^{\infty} n^\beta |a_n|^2\right)\\
\cdot \int_{\DD} \left(\sum_{n= 1}^{\infty} \frac{|z|^{2n}(1-|z|^2)^{1-\beta} }{n^\beta |1-z|^{2a}}\right) dA(z).
\end{multline*}
By Fubini's theorem, it now suffices to show that the series
\begin{equation}
\label{ser}
\sum_{n=1}^{\infty} \frac{1}{n^\beta} \int_\DD \frac{|z|^{2n}(1-|z|^2)^{1-\beta} }{ |1-z|^{2a}}  dA(z)
\end{equation}
converges. We recall the known asymptotics of the Beta function, valid for $t>0$,
\[\int_0^1r^n(1-r)^{t-1} dr = B(n+1, t) \asymp n^{-t};\] 
moreover 
\[\int_0^{2\pi} \frac{1}{1+r^2 - 2r \cos x} dx 
\asymp (1-r)^{-1}.\] 
Applying Jensen's inequality, bearing in mind that $a<1$, we get
\begin{align*}
\int_\DD \frac{|z|^{2n}(1-|z|^2)^{1-\beta} }{ |1-z|^{2a}}  dA(z)
&= \int_0^1r^{2n+1} (1-r^2)^{1-\beta}\\& \,\cdot \left(\int_0^{2\pi}  \frac{dx}{(1+ r^2 - 2r \cos x)^a}\right)dr \\
&\leq C\cdot 
\int_0^1 r^{2n+1 } (1-r)^{1-\beta} (1-r)^{-a} dr \\&
=C\cdot B(n+1, 2-a-\beta)
\asymp n^{a+\beta -2},
\end{align*}
and it follows that \eqref{ser} converges.
\end{proof}

\begin{proof}[Proof of Lemma \ref{L:symmetrization}]
We proceed inductively. The base case $n=1$ follows directly from Lemma \ref{L:changevar}. Let us address the inductive step. Set $$f_{\vec\nu_{n-1}}(z_1,z_2)= q_{\nu_1}(z_1, b_1(z_2))\cdots q_{\nu_{n-1}}(z_1, b_{n-1}(z_2))$$ and $$f_\nu(z_1,z_2) = f_{\vec \nu_{n-1} }(z_1, z_2) q_\nu (z_1, b_n(z_2)).$$ Our task is to show that by choosing $\nu$ large enough, we can bring
$\Dveca(f_\nu)$ 
arbitrarily close to $\Dveca(f_{\vec \nu_{n-1}})$.

Let $N$ be such that $2N-1 >\alpha_1$. To bound the $\Dveca$-norm of $f_\nu$ (where potentially $\alpha_1>1$) we estimate the norms $\|\partial_{z_1}^l f_\nu(0,\cdot)\|_{D_{\alpha_2}}$ for $l\leq N-1$,  the 
two one-variable norms $\|\partial_{z_1}^{N-1}\partial_{z_2}f_\nu(0,\cdot)\|_{D_{\alpha_2}}$ and $\|\partial_{z_1}^N f_\nu (\cdot, 0)\|_{D_{\alpha_1-2N}}$, and the double integral
\begin{equation}\label{eq:N1}
\int_{D\times \DD} |\partial_{z_1}^N \partial_{z_2} f_\nu (z_1, z_2)|^2 dA_{\alpha_1 - 2(N-1)}(z_1) dA_{\alpha_2}(z_2).
\end{equation}
Let us present the details for the double integral \eqref{eq:N1}---the proofs for the the one-variable integrals follow along the same lines.

Computing $\partial_{z_1}^N \partial_{z_2} f_\nu$ we see that it comprises the following terms:
\begin{itemize}
\item[(i)] $\partial_{z_1}^N \partial_{z_2}f_{\vec \nu_{n-1}}(z_1, z_2) q_\nu(z_1, b_n(z_2))$,
\medskip
\item[(ii)] $\partial_{z_1}^k \partial_{z_2} f_{\vec \nu_{n-1}}(z_1, z_2)\partial_{z_1}^l  q_\nu(z_1, b_n(z_2))$, where $k+l=N$ and $l\geq 1$.
\medskip
\item[(iii)] $\partial_{z_1}^k f_{\vec \nu_{n-1}}(z_1, z_2)\partial_{z_1}^l  \partial_{z_2} q_\nu(z_1, b_n(z_2)) b_n'(z_2)$, where $k+l=N$,
\medskip
\end{itemize}
We shall show that (ii) and (iii) tend to $0$ in $L^2 (d A_{\alpha_1-2(N-1)}(z_1) dA_{\alpha_2}(z_2))$ as $\nu\to \infty$. Moreover, we shall show the integral of the modulus squared of (i) with respect to the measure $d A_{\alpha_1-2(N-1)} (z_1) dA_{\alpha_2}(z_2)$ tends to $$\int_{\DD^2} |\partial_{z_1}^N \partial_{z_2} f_{\vec \nu_{n-1}}(z_1, z_2)|^2 d A_{\alpha_1-2(N-1)}(z_1) dA_{\alpha_2}(z_2),$$ which is then controlled by  the induction hypothesis.

We first turn to (iii). Since $\partial_{z_1}^k f_{\vec \nu_{n-1}}$ is bounded on $\DD^2$ we need to estimate $$\int_{\DD^2} |\partial_{z_1}^l \partial_{z_2} q_\nu(z_1, b_n(z_2))b_n'(z_2)|^2 dA_{\alpha_1-2(N-1)} (z_1) dA_{\alpha_2}(z_2),$$ which we do by changing variables (Lemma~\ref{L:changevar}) and invoking the hypothesis that $q_{\nu}\to 1$ in $\Dveca$ as $\nu \to \infty$.

We next continue with terms (i) and (ii) simultaneously. Since $q_\nu(0,b_n(0)$ tends to $1$, we first replace $q_\nu$ with $q_\nu-q_\nu(0,b_n(0))$, and we show that after performing this replacement, the resulting terms converge to $0$ in $L^2(dA_{\alpha_1- 2(N-1)}(z_2) dA_{\alpha_2}(z_2))$.  Note that the functions $\partial_{z_1}^k\partial_{z_2} f_{\vec \nu_{n-1}}(z_1,z_2)$ need not be bounded. However, each of them can be written as a sum of products of a function that is bounded on $\DD^2$, and the functions $b_j'(z_2)\in \mathcal O(D)$, $j=1,\ldots, n-1$. Hence, it suffices to show that 
\[\int_{\DD^2} |b_j'(z_2)|^2  |\partial^l_{z_1} q_\nu(z_1, b_n(z_2))|^2 dA_{\alpha_1-2(N-1)}(z_1) dA_{\alpha_2}(z_2)\] tends to $0$ as $\nu\to \infty$ if $l\leq N$.
Note that the set of singularities $S\cap \overline{\DD}$ of the functions $b_j$, $j=1, \ldots n$, is finite. Invoking the estimate \eqref{puiseuxest} at a point $a\in S$ to bound $|b'_j|$ leads us to integrals of the form 
\[\int_{\DD^2}\frac{1}{|z_2-a|^{2\tau}} |\partial_{z_1}^l q_\nu(z_1, b_n(z_2))|^2 dA_{\alpha_1-2(N-1)}(z_1) dA_{\alpha_2}(z_2),\] 
where $\tau< 1$ and $a\in \overline{\DD}$. We distinguish between two cases: $a\in \DD$ and $a\in \TT$. 

We first deal with the case when $a$ is in the boundary of $\DD$. Applying a rotation if necessary, we can take $a=1$ and use Lemma \ref{L: quotientderivative} (keeping $z_1$ fixed) to get
\begin{multline}\label{mult}\int_\DD \frac{1}{|z_2-a|^{2\tau}}  |\partial_{z_1}^l q_\nu(z_1, b_n(z_2))|^2 dA_{\alpha_2}(z_2)  \\ \leq 
 C\left( |\partial_{z_1}^l q_\nu (z_1, b_n(0))|^2\right)+  C\int_{\DD} |\partial_{z_1}^l \partial_{z_2} q_\nu (z_1, z_2) |^2 dA_{\alpha_2}(z_2).
\end{multline}
We now integrate over $\DD$ with respect to $dA_{\alpha_1-2(N-1)}(z_1)$, and since the right-hand side tends to zero $\nu$, we obtain the desired conclusion.

Now let us consider the case when $a$ lies in the open unit disk. We can cover a neighborhood of $\DD\setminus \{a\}$ with a finite family of relatively compact disks, and estimate the integral over any such disk $\Delta$. This leaves us with two types of integrals, one over $\DD\setminus \Delta$ where $|z_2-a|>c>0$, meaning that it suffices to bound
\begin{equation}\label{eq:kon}\int_{\DD^2} |\partial_{z_1}^l q_\nu(z_1, b_n(z_2))|^2 dA_{\alpha_1-2 (N-1) }(z_1) dA_{\alpha_2}(z_2),
\end{equation}
and the contribution that arises from integrating over $\Delta$,
\begin{equation}\label{eq:kon1}\int_{\Delta} \frac{1}{|z_2-a|^{2\tau}} |\partial_{z_1}^l q_\nu(z_1, b_n(z_2)|^2 dA_{\alpha_1-2(N-1)}(z_1) dA_{\alpha_2}(z_2).
\end{equation}
A bound on \eqref{eq:kon} follows from  \eqref{mult} (with $\tau=0$).
Since $b_n(\Delta)$ is relatively compact in $\DD$ and $(1-|z_2|^2)^{1-\alpha_2}$ is bounded, after an affine change of variables in \eqref{eq:kon1}, we can bound this integral by
$$\int_{\DD^2} \frac{1}{|z_2-1|^{2\tau}} |\partial_{z_1}^l q_\nu (z_1, b(z_2))|^2 dA_{\alpha_1-2(N-1)}(z_1) dA(z_2), $$ where $b\colon \DD\to \DD$ is of bounded multiplicity and $b(\DD)\Subset \DD$. That integral, according to Lemma~\ref{L: quotientderivative}, can be estimated by
\begin{equation}\label{eq:kon2}
\int_{\DD^2}  |\partial_{z_1}^l\partial_{z_2} q_\nu (z_1, b(z_2))|^2 |b'(z)|^2 dA_{\alpha_1-2(N-1)}(z_1) dA(z_2).
\end{equation}
Since $b(\DD)$ is relatively compact in $\DD$ we can insert the weight $(1-|b(z_2)|^2)^{1-\alpha_2}$ in the above integral. Changing variables as in Lemma~\ref{L:changevar}, we bound \eqref{eq:kon2} by 
$$\int_{\DD^2}  |\partial_{z_1}^l\partial_{z_2} q_\nu (z_1, z_2)|^2 dA_{\alpha_1-2(N-1)}(z_1) dA_{\alpha_2}(z_2),$$
which we again have control over. 

This concludes the proof of Lemma \ref{L:symmetrization}.
\end{proof}

\section*{Appendix A}
Here we give a direct proof that $p(z_1,z_2)=2-z_1-z_2$ is cyclic for
$\Dveca$ whenever $\min\{\alpha_1,\alpha_2\}\le1$.

\begin{proof}
Let $f$ be a function in $\Dveca$ orthogonal to $z_1^jz_2^kp(z_1,z_2)$ for all $j,k\ge0$. We need to show that $f=0$. If we write
\[
f(z_1,z_2):=\sum_{k,l\ge0}\frac{b_{k,l}}{(k+1)^{\alpha_1}(l+1)^{\alpha_2}}z_1^kz_2^l,
\]
then the orthogonality condition becomes 
\begin{equation}\label{E:orthog}
2b_{k,l}=b_{l+1,l}+b_{k,l+1} \qquad(k,l\ge0),
\end{equation}
and the condition that $f$ belongs to $\Dveca$ translates to
\begin{equation}\label{E:Dirichlet}
\sum_{k,l\ge0}\frac{|b_{k,l}|^2}{(k+1)^{\alpha_1}(l+1)^{\alpha_2}}<\infty.
\end{equation}
Our goal is thus prove that conditions 
\eqref{E:orthog} and \eqref{E:Dirichlet} together imply that $b_{k,l}=0$ for all $k,l\ge0$.
We can suppose without loss of generality that $\alpha_1\ge1$ and $\alpha_2\ge1$.

Define a new function $g$ by
\[
g(z_1,z_2):=\sum_{k,l\ge0}b_{k,l}z_1^kz_2^l.
\]
Condition \eqref{E:Dirichlet} ensures that $g$ is holomorphic on $\DD^2$.  Also condition \eqref{E:orthog} implies the identity
\[
2g(z_1,z_2)=\frac{g(z_1,z_2)-g(0,z_2)}{z_1}+\frac{g(z_1,z_2)-g(z_1,0)}{z_2} 
\qquad(z \in\DD^2),
\]
which, after rearrangement, becomes
\begin{equation}\label{E:fundamental}
(z_1+z_2-2z_1z_2)g(z_1,z_2)= z_1g(z_1,0)+z_2g(0,z_2) 
\qquad(z \in\DD^2).
\end{equation}
Consider now the substitutions $z_1:=\zeta/(\zeta-1)$ and $z_2:=\zeta/(\zeta+1)$. Note that we have $z_1\in\DD\iff \Re\zeta<1/2$ and  $z_2\in\DD\iff \Re\zeta>-1/2$. Substituting these values of $z_1,z_2$ into \eqref{E:fundamental}, we find that
\[
0=\frac{\zeta}{\zeta-1}g\Bigl(\frac{\zeta}{\zeta-1},0\Bigr)+
\frac{\zeta}{\zeta+1}g\Bigl(0,\frac{\zeta}{\zeta+1}\Bigr)
\qquad(-1/2<\Re\zeta<1/2).
\]
Thus, if we define $h\colon \CC\to\CC$ by
\[
h(\zeta):=
\begin{cases}
\frac{1}{\zeta-1}g\Bigl(\frac{\zeta}{\zeta-1},0\Bigr),
&\Re\zeta<1/2,\\
\frac{-1}{\zeta+1}g\Bigl(0,\frac{\zeta}{\zeta+1}\Bigr),
&\Re\zeta>-1/2,
\end{cases}
\]
then $h$ is a well-defined entire function. Note also that
\begin{align*}
\sum_{k\ge0}\frac{|b_{k0}|^2}{(k+1)^{\alpha_1}}
&\asymp\int_\DD |g(z_1,0)|^2(1-|z_1|^2)^{\alpha_1-1}\,dA(z_1)\\
&=\int_{\Re\zeta<1/2} |(\zeta-1)h(\zeta)|^2\Bigl(1-\Bigl|\frac{\zeta}{\zeta-1}\Bigr|^2\Bigr)^{\alpha_1-1}\,
\frac{dA(\zeta)}{|\zeta-1|^4}\\
&=\int_{\Re\zeta<1/2}|h(\zeta)|^2\frac{(1-2\Re\zeta)^{\alpha_1-1}}{|\zeta-1|^{2\alpha_1}}\,dA(\zeta),
\intertext{and likewise}
\sum_{l\ge0}\frac{|b_{0l}|^2}{(l+1)^{\alpha_2}}
&\asymp\int_\DD |g(0,z_2)|^2(1-|z_2|^2)^{\alpha_2-1}\,dA(z_2)\\
&=\int_{\Re\zeta>-1/2} |(\zeta+1)h(\zeta)|^2\Bigl(1-\Bigl|\frac{\zeta}{\zeta+1}\Bigr|^2\Bigr)^{\alpha_2-1}\, \frac{dA(\zeta)}{|\zeta+1|^4}\\
&=\int_{\Re\zeta>-1/2}|h(\zeta)|^2\frac{(1+2\Re\zeta)^{\alpha_2-1}}{|\zeta+1|^{2\alpha_2}}\,dA(\zeta).
\end{align*}
Both these series are finite, by \eqref{E:Dirichlet}, so the sum of the two integrals is finite, and consequently
\[
\int_{|\zeta|>1} \frac{|h(\zeta)|^2}{|\zeta|^{2M}}\,dA(\zeta)<\infty,
\]
where $M:=\max\{\alpha_1,\alpha_2\}$. This forces $h$ to be a polynomial. Thus, if $h\not\equiv0$, then $|h(\zeta)|\ge c>0$ for all large $|\zeta|$, and substituting this information back into the integrals already known to be finite, we get
\[
\int_{\Re\zeta<1/2}\frac{(1-2\Re\zeta)^{\alpha_1-1}}{|\zeta-1|^{2\alpha_1}}\,dA(\zeta)<\infty\]
and
\[\int_{\Re\zeta>-1/2}\frac{(1+2\Re\zeta)^{\alpha_2-1}}{|\zeta+1|^{2\alpha_2}}\,dA(\zeta)<\infty,
\]
which implies that $\alpha_1> 1$ and $\alpha_2>1$, contrary to the hypothesis that $\min\{\alpha_1,\alpha_2\}\le1$.
We conclude that $h\equiv0$.  It follows that $g(z_1,0)\equiv0$ and $g(0,z_2)\equiv0$, whence also $g(z_1,z_2)\equiv0$ by \eqref{E:fundamental} and so, finally, $b_{k,l}=0$ for all $k,l\ge0$.
\end{proof}
\begin{remark} 
The proof shows that, if $\min\{\alpha_1,\alpha_2\}\le1$, then  the conditions
\[
2b_{k,l}=b_{l+1,l}+b_{k,l+1} \qquad(k,l\ge0)
\]
and
\[
\sum_{k\ge0}\frac{|b_{k0}|^2}{(k+1)^{\alpha_1}}<\infty
\quad\text{and}\quad \sum_{l\ge0}\frac{|b_{0l}|^2}{(l+1)^{\alpha_2}}<\infty
\]
together imply that $b_{k,l}=0$ for all $j,k\ge0$.
\end{remark}
\section*{Appendix B}

\begin{lemma}[Derivative estimate for roots]\label{rootglue}
Suppose \[P(z_1,z_2)=z_2^n + A_1(z_1) z_2^{n-1}+ \cdots +A_n(z_1)\] is holomorphic in a domain 
$G\times \CC\subset \CC^2$. If $h\in \mathrm{Hol}(G')$ for a domain $G'\Subset  G$, 
and satisfies $P(z_1, h(z_1))\equiv 0$, then 
\[|h'(z_1)|\leq O(|z_1-\omega|^{1-1/n})\] 
as $z_1\to \omega\in \partial G'$. 
\end{lemma}
This result is probably known, but we were not able to locate it in the literature, and hence we include its proof.
\begin{proof}
We may assume $\omega =0$ and $h(0)=0$. Losing no generality we may demand that an analytic set $\{P=0\}$ is irreducible in a neighborhood of $0$ (otherwise we may decrease $n$). Applying Puiseux parametrization (see \cite{LojBook}, Corollary, page 171) we find that there is an analytic germ in a neighborhood of $0$ such that $\{P=0\} = \{(\la^n, \phi(\la))\}$ in a neighborhood of $0\in \CC^2$. This means that $h(z_1)= \phi(z_1^{1/n})$, where the branch of square is properly chosen. From this one can immediately derive the assertion.
\end{proof}

\end{document}